\documentclass[11pt]{article}

\usepackage{amsmath,amssymb,amsthm,amsfonts}
\usepackage{mathtools}
\usepackage{mathrsfs}
\usepackage{bm}
\usepackage{graphicx}
\usepackage{float}
\usepackage{multirow}
\usepackage{booktabs}
\usepackage{longtable}
\usepackage{arydshln}
\usepackage{enumitem}
\usepackage{cite}
\usepackage{microtype}
\usepackage[a4paper,margin=2.5cm]{geometry}
\usepackage{hyperref}

\hypersetup{
	colorlinks=true,
	linkcolor=blue,
	citecolor=blue,
	urlcolor=blue
}

\allowdisplaybreaks[2]
\numberwithin{equation}{section}

\theoremstyle{plain}
\newtheorem{theorem}{Theorem}[section]
\newtheorem{lemma}[theorem]{Lemma}

\newtheorem{corollary}[theorem]{Corollary}

\theoremstyle{definition}

\theoremstyle{remark}
\newtheorem{remark}[theorem]{Remark}

\newcommand{\C}{\mathbb C}
\newcommand{\R}{\mathbb R}

\newcommand{\Focka}{  F^2_\alpha}
\newcommand{\Fockone}{\mathcal F^2_1}
\newcommand{\Hol}{\operatorname{Hol}}
\newcommand{\Fsa}{\mathcal F_{\alpha}}
\newcommand{\Kern}{\mathcal K}
\newcommand{\TopAlg}{\mathcal T_\alpha}
\newcommand{\Pol}{\mathcal P}

\title{\Large\bf
A Fourier Criterion for the Toeplitzness of Operators on Fock Spaces
	\footnotetext{\hspace{-0.35cm}
		\endgraf{\it E-mail: zhaopeng.lin@mail.dlut.edu.cn (Zhaopeng Lin)}
		\endgraf \hspace{1.1cm} {\it lyfdlut@dlut.edu.cn (Yufeng Lu)}
		\endgraf \hspace{1.1cm} {\it zuchao@dlut.edu.cn (Chao Zu)}
		\endgraf Y. Lu was supported by the National Natural Science Foundation of China
		(Grant No. 12031002). C. Zu was supported by the National Natural Science
		Foundation of China (Grant No. 12401151), and the Postdoctoral Researcher
		Foundation of China (Grant No. GZB20240100).
}}
 
\author{Zhaopeng Lin, Yufeng Lu, Chao Zu\thanks{Corresponding author.}}
\date{}

\begin{document}
	
	\maketitle
	
	\vspace{-0.8cm}
	
	\begin{center}
		\begin{minipage}{14cm}\small
			\noindent{\bf Abstract.}\quad
We give a Fourier criterion for the Toeplitzness of bounded operators on Fock
spaces, where Toeplitzness means representability as a  Toeplitz operator
with a bounded measurable symbol.  For a Toeplitz operator, the anti-diagonal
restriction of its canonical kernel is the Fourier transform of the
Gaussian-weighted symbol.  Consequently, Fourier inversion of this anti-diagonal
restriction recovers the unique bounded symbol whenever such a representation
exists.  As applications, we characterize the Toeplitzness  of weighted composition operators and generalized
Volterra-type operators.
			\endgraf
			\endgraf
		\end{minipage}
	\end{center}

\section{Introduction}

Let \(\mathcal H\subset L^2(X,\mu)\) be a closed reproducing kernel Hilbert
space of analytic functions, and let \(P_{\mathcal H}\) denote the orthogonal
projection of \(L^2(X,\mu)\) onto \(\mathcal H\).  For a symbol \(f\) in a
prescribed function class, the Toeplitz operator with symbol \(f\) is
\[
T_f=P_{\mathcal H}M_f|_{\mathcal H},
\]
where \(M_f\) denotes multiplication by \(f\).   Toeplitz operators constitute one of
the most important model classes in the theory of non-selfadjoint operators.
Their study lies at the intersection of operator theory, function 
theory, and the theory of Banach algebras.

  On Hardy and Bergman spaces, Toeplitz operators provide a canonical mechanism for translating the   geometry and function theory of the symbol to 
  operator-theoretic properties of the induced operator, such as spectrum,
  Fredholmness, index, and invertibility. The classical Hardy-space 
  theory gives a particularly transparent example: if
  \(f\in C(\mathbb T)\), then \(T_f\) is Fredholm if and only if \(f\) has no
  zeros on \(\mathbb T\).  
Thus a topological invariant of the symbol becomes the Fredholm index of the 
operator; see \cite{BrownHalmos,MR361893,MR2223704}.   
On Fock spaces, Toeplitz operators arise naturally in classical models of  quantization.  This connects  Toeplitz 
operators with coherent states, quantum mechanics, Weyl calculus, and quantum 
harmonic analysis; see 
\cite{MR350504,MR411452,MR859136,Folland,MR2934601}.

A natural problem is then the Toeplitzness problem: given a bounded operator
\(A\) on \(\mathcal H\), when does there exist a symbol \(f\) in the prescribed class
such that
\[
A=T_f?
\]
We shall use the term  Toeplitzness for this property.

The Toeplitzness problem originates in the classical theorem of Brown and
Halmos.  On the Hardy space \(H^2(\mathbb D)\), if \(S\) is the unilateral shift,
then a bounded operator \(A\) is Toeplitz if and only if
\[
S^*AS=A.
\]
Equivalently, the matrix of \(A\) with respect to the monomial basis has constant
diagonals.  This identity gives a purely operator-theoretic criterion for the
Toeplitzness of \(A\) and has served as the model for later Toeplitzness criteria
on analytic function spaces; see \cite{BrownHalmos}.  For concrete operator
classes, Nazarov and Shapiro showed
that a composition operator \(C_\varphi\) on \(H^2(\mathbb D)\) is Toeplitz 
only in the trivial case \(\varphi(z)=z\); 
see \cite{NazarovShapiro}.  For weighted composition operators, Ohno proved 
that the only nonzero Toeplitz weighted composition operators on 
\(H^2(\mathbb D)\) are multiplication operators; see \cite{Ohno}.

On Bergman spaces the situation is subtler: the Bergman shift is a weighted
shift, so the Hardy-space identity \(S^*AS=A\) has to be replaced by identities
adapted to the weighted Bergman shift and its Cauchy dual.  Such
Brown--Halmos type criteria are known in particular for Toeplitz operators with
bounded harmonic symbols; see
\cite{MR1867348,LouhichiOlofsson,OlofssonWennman}. These results have
also been applied to concrete operators.  Manhas and Zhao
proved that weighted composition operators on weighted Bergman spaces which
are Toeplitz operators with bounded harmonic symbols are precisely
multiplication operators; see \cite{ManhasZhao}.

The Fock space has a different structure: multiplication by the coordinate function is an unbounded operator on
\(\Focka\).  Hence the classical Brown--Halmos shift identity has no literal
bounded-shift analogue in this setting.  The main purpose
of this paper is to replace such a missing shift equation by a Fourier-theoretic
Toeplitzness criterion.  More precisely, we show that the anti-diagonal restriction
of the canonical kernel of a Toeplitz operator is a Fourier transform of its
Gaussian-weighted symbol.  Thus the Toeplitzness of a bounded operator  can be tested by Fourier inversion of this
anti-diagonal restriction.
 
A closely related line of work concerns the Fock--Toeplitz algebra.  Let
\(\TopAlg\) denote the \(C^*\)-algebra generated by Toeplitz operators with
bounded measurable symbols on \(\Focka\).  Bauer, Fulsche and Rodriguez
Rodriguez studied several concrete classes of operators from this perspective,
including weighted composition operators and generalized Volterra--type
operators; see \cite{BFRR}.  Following the terminology of \cite{BFRR}, we
distinguish the membership problem \textup{(A)} from the representation problem
\textup{(B)}.  Problem \textup{(A)} asks whether a given concrete operator
belongs to \(\TopAlg\).  Problem \textup{(B)} asks whether it is itself a single
Toeplitz operator with a bounded measurable symbol and, in that case, how to
recover the symbol.  The latter problem is more rigid.

For weighted composition operators, \cite[Question~1]{BFRR} asks when a compact
operator of this type is a Toeplitz operator, and
\cite[Theorem~3.12]{BFRR} gives a sufficient condition with an explicit bounded
symbol.  Our Theorem~\ref{thm:weighted} gives the corresponding necessary and
sufficient condition, and hence settles the bounded-symbol Toeplitzness problem
for weighted composition operators.

For generalized Volterra--type operators, \cite[Theorem~3.25]{BFRR} gives a
complete answer to the membership problem \textup{(A)} in the Fock--Toeplitz
algebra.  We address here the stronger representation problem \textup{(B)} for
the same class.  Theorem~\ref{thm:volterra-main} determines exactly when such
an operator is itself a Toeplitz operator with a bounded measurable symbol, and
in that case recovers the unique bounded symbol explicitly.

Let \( F^2_\alpha\) denote the weighted Fock space associated with
\[
d\lambda_\alpha(z)=\frac{\alpha}{\pi}e^{-\alpha|z|^2}\,dA(z),
\qquad \alpha>0,
\]
where \(dA\) is Lebesgue area measure.  The reproducing kernel is $
K_w(z)$ for $w\in \mathbb{C}$. 
For a bounded operator \(A\) on \(\mathcal F^2_\alpha\), its canonical kernel is
\[
\Kern_A(w,z)=\langle AK_w,K_z\rangle_\alpha .
\]
Set
\[
H_0^A(w):=\Kern_A(w,-w),
\qquad
H_1^A(w):=\left.\partial_z\Kern_A(w,z)\right|_{z=-w}.
\]

We use the   Fourier transform
\[
(\mathcal F_{ \alpha}G)(w)
=
\int_{\mathbb C}G(u)
e^{\alpha(u\bar w-\bar u w)}\,dA(u).
\]
This is the ordinary Euclidean Fourier transform after a fixed rotation and
dilation of the frequency variable.
 
 We now state the main Fourier criterion of the paper.  It characterizes when a
 bounded operator on \(\Focka\) is a Toeplitz operator with a bounded measurable
 symbol, and, in that case, recovers the unique bounded symbol by Fourier inversion
 of the anti-diagonal restriction.

\begin{theorem}
	\label{thm:higher-recovery}
	Let $A\in\mathcal L(\Focka)$. Then the following assertions hold.
	
	\begin{enumerate}[label=\textup{(\alph*)}]
	\item 
	The operator \(A\) is of the form \(A=T_f\) with \(f\in L^\infty(\C,dA)\)
	if and only if \(H_0^A\in L^2(\C,dA)\) and the function defined a.e. by
	\begin{equation}\label{eq:recovered-symbol-zero}
		f_A^{(0)}(u)
		:=
		\frac{\pi}{\alpha}
		e^{\alpha|u|^2}
		\bigl(\Fsa^{-1}H_0^A\bigr)(u)
	\end{equation}
	belongs to \(L^\infty(\C,dA)\).  In that case,
	\[
	A=T_{f_A^{(0)}},
	\]
	and the bounded Toeplitz symbol is unique almost everywhere.
		
	\item 
	The operator \(A\) is of the form \(A=T_f\) for some
	\(f\in L^\infty(\C,dA)\) if and only if \(H_1^A\in L^2(\C,dA)\), the a.e. defined 
	function on \(\C\) given by
	\begin{equation}\label{eq:first-order-recovered-symbol}
		f_A^{(1)}(u)
		:=
		\frac{1}{\bar u}
		\frac{\pi}{\alpha^2}
		e^{\alpha|u|^2}
		\bigl(\Fsa^{-1}H_1^A\bigr)(u), 
		\qquad u\ne0.
	\end{equation}
belongs to \(L^\infty(\C,dA)\),  and 
	satisfies the origin functional identity
	\begin{equation}\label{eq:origin-functional-condition}
		Ah(0)=T_{f_A^{(1)}}h(0),
		\qquad h\in\Focka .
	\end{equation}
	In that case,
	\[
	A=T_{f_A^{(1)}},
	\]
	and the bounded Toeplitz symbol is unique almost everywhere.
	\end{enumerate}
\end{theorem}
 
This anti-diagonal approach should be compared with the usual  Berezin
transform.  For a bounded symbol \(f\), the Berezin transform of \(T_f\) is the
Gaussian heat transform of \(f\):
\[
\widetilde{T_f}(z)
=
\frac{\alpha}{\pi}\int_{\C}f(u)e^{-\alpha|u-z|^2}\,dA(u).
\]
Thus the Berezin transform  of a Toeplitz operator give a necessary heat-transform
condition on the symbol.  However, recovering \(f\) from \(\widetilde{T_f}\)
requires a backward heat operation, which is an unstable inverse problem.  

By contrast, the anti-diagonal restriction turns the canonical Toeplitz kernel
directly into a Fourier transform of the Gaussian-weighted symbol.  The
bounded-symbol recognition problem is therefore reduced to Fourier inversion
and an \(L^\infty\)-test for the recovered symbol.  This is the main advantage of
the anti-diagonal criterion: it produces both the only possible symbol and, by
the anti-diagonal uniqueness lemma, the sufficiency of the representation.
This phase-space viewpoint is natural in Fock analysis and is closely related
to the Bargmann transform, Weyl calculus, and quantum harmonic analysis; see
\cite{Folland,BFRR}.

Our first application concerns the Toeplitzness of weighted composition operators
\[
W_{\psi,\varphi}h=\psi\cdot(h\circ\varphi).
\]
Here \(\psi,\varphi\in\Hol(\C)\), where \(\Hol(\C)\) denotes the space of
entire functions on \(\C\). 
Assume that \(W_{\psi,\varphi}=T_f\) for some \(f\in L^\infty(\C)\) and
\(\psi\not\equiv0\).  Then \(W_{\psi,\varphi}\) is bounded, and the standard
boundedness theorem for weighted composition operators first reduces the
composition symbol to the affine form
\[
\varphi(z)=a+\lambda z,\qquad |\lambda|\le1.
\]
The zeroth-order anti-diagonal recovery theorem then gives the only possible
bounded Toeplitz symbol.  This turns the sufficient criterion of
\cite[Theorem~3.12]{BFRR} into a complete characterization.  For
\(\lambda\ne0\), the candidate symbol is
\[
F_{\psi,a,\lambda}(w)
=
\frac1\lambda
\exp\!\left[
\alpha\left(
\frac{\lambda-1}{\lambda}|w|^2+\frac a\lambda\bar w
\right)\right]
\psi\!\left(\frac{w-a}{\lambda}\right).
\]
\begin{theorem}\label{thm:weighted}
	Let $\psi,\varphi\in\Hol(\C)$.  The following are equivalent.
	\begin{enumerate}[label=\textup{(\roman*)}]
		\item There exists $f\in L^\infty(\C,dA)$ such that $
		W_{\psi,\varphi}=T_f$ 
		on $\Focka$.
		
		\item Either $\psi\equiv0$, or there exist $a,\lambda\in\C$ such that
		\[
		\varphi(z)=a+\lambda z,\qquad \lambda\ne0,
		\]
		and
		\begin{equation}\label{eq:weighted-disk-alpha}
			\Re\lambda\ge|\lambda|^2,
		\end{equation}
		and
		\begin{equation}\label{eq:weighted-multiplier-condition-alpha}
			\sup_{z\in\C}|\psi(z)|
			\exp\!\left[
			\alpha\left(
			-(\Re\lambda-|\lambda|^2)|z|^2
			+\Re((2\lambda-1)\bar a z)
			\right)\right]<\infty.
		\end{equation}
	\end{enumerate}
Moreover, in the nonzero case \(\psi\not\equiv0\), the bounded Toeplitz symbol
is unique and equals
\[
f=F_{\psi,a,\lambda}
\qquad\text{a.e.}
\]

\end{theorem}

Our second application concerns generalized Volterra-type operators
\[
V_{(g,\varphi)}h(z)
=
\int_0^z h(\varphi(\zeta))g'(\zeta)\,d\zeta .
\]
For these operators, the first-order recovery theorem is the natural tool, since
\[
\partial_z\Kern_{V_{(g,\varphi)}}(w,z)
=
g'(z)e^{\alpha\varphi(z)\bar w}.
\]
After the affine reduction and the origin functional condition, the nonzero
Toeplitz case is forced to be centered:
\[
\varphi(z)=\lambda z,\qquad \lambda\ne0.
\]
In this case the first-order recovery theorem gives the candidate symbol
\[
\sigma_{\lambda,g}(w)
=
\frac{1}{\alpha\lambda\bar w}
\exp\!\left[
\alpha\frac{\lambda-1}{\lambda}|w|^2
\right]
g'\!\left(\frac w\lambda\right),
\qquad w\ne0.
\]
\begin{theorem} \label{thm:volterra-main}
	Let $g,\varphi\in\Hol(\C)$.  The following are equivalent.
	\begin{enumerate}[label=\textup{(\roman*)}]
		\item There exists $f\in L^\infty(\C,dA)$ such that $
		V_{(g,\varphi)}=T_f$ 
		on $\Focka$.
		\item Either $g$ is constant, or there exists
		$\lambda\in\C\setminus\{0\}$ such that $
		\varphi(z)=\lambda z$ 
		and $
		\sigma_{\lambda,g}\in L^\infty(\C,dA)$. 
	\end{enumerate}
In the nonzero case \(g'\not\equiv0\), the bounded symbol is unique and equals
\(\sigma_{\lambda,g}\) almost everywhere.  Moreover, the condition
\(\sigma_{\lambda,g}\in L^\infty(\C,dA)\) automatically implies
\begin{equation}\label{eq:volterra-automatic-conditions-alpha}
	g'(0)=0,
	\qquad
	\Re\lambda\ge|\lambda|^2>0.
\end{equation}
\end{theorem}

The paper is organized as follows.  Section~2 
introduces the Fourier transform, and establishes the anti-diagonal recovery
criteria.  Section~3 applies these criteria to weighted composition operators
and generalized Volterra-type operators.
\section{Anti-diagonal symbol recovery}

\subsection{Fock space, Toeplitz operators, and canonical kernels}

Fix $\alpha>0$ and put
\[
d\lambda_\alpha(z)=\frac\alpha{\pi}e^{-\alpha|z|^2}\,dA(z),
\qquad
\Focka=\Hol(\C)\cap L^2(\C,d\lambda_\alpha),
\]
where $dA$ denotes Lebesgue area measure.  The inner product is linear in the
first variable:
\[
\langle h,k\rangle_\alpha
=
\int_\C h(z)\overline{k(z)}\,d\lambda_\alpha(z).
\]
The reproducing kernel and its norm are
\[
K_w(z)=e^{\alpha z\bar w},
\qquad
\|K_w\|_\alpha=e^{\alpha|w|^2/2}.
\]

For $f\in L^\infty(\C,dA)$, define
\[
T_fh(z)
=
\frac\alpha{\pi}\int_\C
f(u)h(u)e^{\alpha z\bar u-\alpha|u|^2}\,dA(u).
\]
Let $\TopAlg$ denote the $C^*$-algebra on $\Focka$ generated by all $T_f$ with
$f\in L^\infty(\C,dA)$.  

For a bounded operator $A$ on $\Focka$, its canonical kernel is
\[
\Kern_A(w,z)
=
\langle AK_w,K_z\rangle_\alpha.
\]
Then
\[
Ah(z)=\int_\C \Kern_A(w,z)h(w)\,d\lambda_\alpha(w),
\]
and
\begin{equation}\label{eq:toeplitz-kernel-alpha}
	\Kern_{T_f}(w,z)
	=
	\frac\alpha{\pi}\int_\C
	f(u)e^{-\alpha|u|^2+\alpha u\bar w+\alpha z\bar u}\,dA(u).
\end{equation}
The function $\Kern_A$ is anti-entire in $w$ and entire in $z$.

Set
\begin{align}
	H_0^A(w)&:=\Kern_A(w,-w),\\
	H_1^A(w)&:=\left.\partial_z\Kern_A(w,z)\right|_{z=-w}.
\end{align}
In particular, \(H_0^A\) and \(H_1^A\) are continuous functions on \(\C\).

The Berezin transform of $A$ is
\[
\widetilde A(z)=\langle Ak_z,k_z\rangle_\alpha.
\]
For a bounded symbol $f$,
\begin{equation}\label{eq:berezin-symbol-alpha}
	\widetilde{T_f}(z)
	=
	\widetilde f(z)
	=
	\frac\alpha{\pi}
	\int_\C f(u)e^{-\alpha|u-z|^2}\,dA(u).
\end{equation} 
\subsection{Unitary change of Gaussian scale}

\begin{lemma} \label{lem:dilation}
Put $
\rho_\alpha=\sqrt{\alpha}$. 	The map
	\[
	U_\alpha:\Focka\longrightarrow\Fockone,
	\qquad
	(U_\alpha h)(z)=h\!\left(\frac z{\rho_\alpha}\right),
	\]
	is unitary.  If
	\[
	f_\alpha(z)=f\!\left(\frac z{\rho_\alpha}\right),\qquad
	\psi_\alpha(z)=\psi\!\left(\frac z{\rho_\alpha}\right),\qquad
	\varphi_\alpha(z)=\rho_\alpha\varphi\!\left(\frac z{\rho_\alpha}\right),
	\]
	and
	\[
	g_\alpha(z)=g\!\left(\frac z{\rho_\alpha}\right),
	\]
	then
	\begin{align}
		U_\alpha T_f U_\alpha^{-1}
		&=T_{f_\alpha}^{(1)},\label{eq:dilation-toeplitz}\\
		U_\alpha W_{\psi,\varphi}U_\alpha^{-1}
		&=W_{\psi_\alpha,\varphi_\alpha},\label{eq:dilation-weighted}\\
		U_\alpha V_{(g,\varphi)}U_\alpha^{-1}
		&=V_{(g_\alpha,\varphi_\alpha)}.\label{eq:dilation-volterra}
	\end{align}
	In particular, boundedness, compactness and membership in the Toeplitz
	algebra  are invariant under this
	change of scale.
\end{lemma}

\begin{proof}
	The change of variables $z=\rho_\alpha u$ gives
	\[
	\|U_\alpha h\|_1^2
	=
	\frac1\pi\int_\C
	\left|h\!\left(\frac z{\rho_\alpha}\right)\right|^2e^{-|z|^2}\,dA(z)
	=
	\frac\alpha{\pi}\int_\C |h(u)|^2e^{-\alpha|u|^2}\,dA(u).
	\]
	Thus $U_\alpha$ is unitary.  Since it intertwines the orthogonal projections onto
	the two Fock spaces, \eqref{eq:dilation-toeplitz} follows from
	$T_f=P_\alpha M_f$.  The identities
	\eqref{eq:dilation-weighted} and \eqref{eq:dilation-volterra} follow by direct
	substitution; for the latter, use
	$g_\alpha'(z)=\rho_\alpha^{-1}g'(z/\rho_\alpha)$.
\end{proof}

\subsection{The Fourier transform}

For $G\in L^1(\C,dA)$, define
\begin{equation}\label{eq:symplectic-fourier-alpha}
	(\Fsa G)(w)
	=\int_\C G(u)e^{\alpha(u\bar w-\bar u w)}\,dA(u).
\end{equation}
Write $u=x+iy$ and $w=s+it$.  Then
\[
\alpha(u\bar w-\bar u w)
=2i\alpha(ys-xt).
\]
Let $\mathcal F$ denote the Euclidean Fourier transform on $\R^2$ with $X=(x,y)$ 
\[
(\mathcal F G)(\xi)
=\int_{\R^2}G(X)e^{-2\pi iX\cdot\xi}\,dx dy.
\]
Folland  records that $\mathcal F $ is an isometry on
$L^2$; see \cite[Prologue, p.~5]{Folland}.   Write
\[
2i\alpha(ys-xt)
=
-2\pi i X\cdot
\left(\frac{\alpha t}{\pi},-\frac{\alpha s}{\pi}\right).
\]
Hence
\[
(\Fsa G)(s+it)
=
(\mathcal F G)\left(\frac{\alpha t}{\pi},
-\frac{\alpha s}{\pi}\right).
\]
By Plancherel for $\mathcal F $ and the change of variables
\[
\xi_1=\frac{\alpha t}{\pi},
\qquad
\xi_2=-\frac{\alpha s}{\pi},
\]
we obtain
\begin{equation}\label{eq:plancherel-alpha}
	\|\Fsa G\|_{L^2(\C,dA)}
	=\frac{\pi}{\alpha}\|G\|_{L^2(\C,dA)}.
\end{equation}
Consequently,
\begin{equation}\label{eq:inverse-symplectic-alpha}
	(\Fsa^{-1}H)(u)
	=\frac{\alpha^2}{\pi^2}\int_\C
	H(w)e^{-\alpha u\bar w+\alpha\bar u w}\,dA(w),
\end{equation}
first for Schwartz functions and then, by Plancherel, on $L^2(\C,dA)$.

For
\[
G_{f,\alpha}(u)
=
\frac\alpha{\pi}f(u)e^{-\alpha|u|^2},
\]
formula \eqref{eq:toeplitz-kernel-alpha} gives
\begin{equation}\label{eq:anti-diagonal-toeplitz-alpha}
	\Kern_{T_f}(w,-w)
	=
	(\Fsa G_{f,\alpha})(w).
\end{equation}

\begin{lemma}\label{lem:anti-diagonal-uniqueness}
	Let \(F(w,z)\) be anti-entire in \(w\) and entire in \(z\). If
	\[
	F(w,-w)=0\qquad (w\in\mathbb C),
	\]
	then \(F\equiv 0\) on \(\mathbb C^2\).
\end{lemma}

\begin{proof}
	Define
	\[
	H(z,\zeta)=F(-\overline{\zeta},z),\qquad (z,\zeta)\in\mathbb C^2.
	\]
	Since \(F\) is anti-entire in the first variable and entire in the second
	variable, \(H\) is entire in the two complex variables \(z\) and \(\zeta\).
	For every \(z\in\mathbb C\), using the hypothesis with \(w=-z\), we get
	\[
	H(z,\bar z)=F(-z,z)=0.
	\]
	By the diagonal uniqueness theorem for holomorphic functions; see \cite{MR1162310}, applied to
	the diagonal \(\{(z,\zeta):\zeta=\bar z\}\), it follows that \(H\equiv 0\)
	on \(\mathbb C^2\). Hence, for arbitrary \(w,z\in\mathbb C\),
	\[
	F(w,z)=H(z,-\bar w)=0.
	\]
	Therefore \(F\equiv 0\).
\end{proof}

\subsection{The main Fourier criterion for Toeplitzness}

\begin{proof}[Proof of Theorem~\ref{thm:higher-recovery}]
	We first prove the necessity.  Suppose that \(A=T_f\) with
	\(f\in L^\infty(\C,dA)\).  From \eqref{eq:toeplitz-kernel-alpha},
	\[
	\Kern_{T_f}(w,-w)
	=
	\frac\alpha{\pi}
	\int_\C
	f(u)e^{-\alpha|u|^2+\alpha(u\bar w-\bar u w)}\,dA(u).
	\]
	Thus
	\begin{equation}\label{eq:H0Toeplitz}
		H_0^{T_f}
		=
		\Fsa\!\left(
		\frac\alpha{\pi}
		f(u)e^{-\alpha|u|^2}
		\right).
	\end{equation}
	Differentiating \eqref{eq:toeplitz-kernel-alpha} with respect to the
	holomorphic kernel variable gives
	\[
	\partial_z\Kern_{T_f}(w,z)
	=
	\frac\alpha{\pi}
	\int_\C
	\alpha\bar u
	f(u)e^{-\alpha|u|^2+\alpha u\bar w+\alpha z\bar u}\,dA(u).
	\]
	The differentiation is justified locally uniformly in \(z\) by
	Cauchy--Schwarz and Gaussian decay.  Setting \(z=-w\), we obtain
	\begin{equation}\label{eq:H1Toeplitz}
		H_1^{T_f}
		=
		\Fsa\!\left(
		\frac{\alpha^2}{\pi}
		\bar u f(u)e^{-\alpha|u|^2}
		\right).
	\end{equation}
	Since
	\[
	f(u)e^{-\alpha|u|^2},
	\qquad
	\bar u f(u)e^{-\alpha|u|^2}
	\]
	belong to \(L^1(\C,dA)\cap L^2(\C,dA)\), Plancherel inversion for \(\Fsa\) gives
	\[
	\Fsa^{-1}H_0^{T_f}(u)
	=
	\frac\alpha{\pi}f(u)e^{-\alpha|u|^2}
	\]
	and
	\[
	\Fsa^{-1}H_1^{T_f}(u)
	=
	\frac{\alpha^2}{\pi}\bar u f(u)e^{-\alpha|u|^2}
	\]
	almost everywhere.  Therefore
	\[
	f_A^{(0)}(u)=f(u)
	\quad\text{a.e.}
	\]
	and, for a.e. \(u\in\C\setminus\{0\}\),
	\[
	f_A^{(1)}(u)=f(u).
	\]
	Hence
	\[
	f_A^{(0)},\ f_A^{(1)}\in L^\infty(\C,dA).
	\]
	Moreover, since \(f_A^{(1)}=f\) a.e., we have
	\[
	T_{f_A^{(1)}}=T_f.
	\]
	Thus
	\[
	Ah(0)=T_fh(0)=T_{f_A^{(1)}}h(0),
	\qquad h\in\Focka.
	\]
	This proves the necessity in both \textup{(a)} and \textup{(b)}.
	
	We now prove the sufficiency in \textup{(a)}.  Assume that
	\(H_0^A\in L^2(\C,dA)\) and that \(f_A^{(0)}\in L^\infty(\C,dA)\).  Put
	\[
	G=\Fsa^{-1}H_0^A.
	\]
	By the definition of \(f_A^{(0)}\),
	\[
	G(u)
	=
	\frac\alpha{\pi}f_A^{(0)}(u)e^{-\alpha|u|^2}
	\qquad\text{a.e.}
	\]
	Hence \(G\in L^1(\C,dA)\cap L^2(\C,dA)\).  Therefore \(\Fsa G\) is continuous, and
	Plancherel gives
	\[
	\Fsa G=H_0^A
	\qquad\text{a.e. on }\C.
	\]
	Since \(H_0^A\) is continuous, the equality holds everywhere.  By
	\eqref{eq:toeplitz-kernel-alpha},
	\[
	\Kern_A(w,-w)
	=
	\Kern_{T_{f_A^{(0)}}}(w,-w),
	\qquad w\in\C.
	\]
	Lemma~\ref{lem:anti-diagonal-uniqueness} gives equality of the two canonical
	kernels on \(\C^2\). Hence
	\[
	A=T_{f_A^{(0)}}.
	\]
	
	We next prove the sufficiency in \textup{(b)}.  Assume that
	\(H_1^A\in L^2(\C,dA)\), that the a.e. defined function \(f_A^{(1)}\) given by
	\eqref{eq:first-order-recovered-symbol} belongs to \(L^\infty(\C,dA)\), and that the origin functional identity
	\[
	Ah(0)=T_{f_A^{(1)}}h(0),
	\qquad h\in\Focka,
	\]
	holds.  Choose an \(L^\infty\)-representative of \(f_A^{(1)}\), and denote it
	by \(f\).  Since Toeplitz operators with bounded symbols depend only on the
	a.e. equivalence class of the symbol, we have
	\[
	T_f=T_{f_A^{(1)}}.
	\]
	By the definition of \(f_A^{(1)}\), we have
	\[
	\Fsa^{-1}H_1^A(u)
	=
	\frac{\alpha^2}{\pi}
	\bar u f(u)e^{-\alpha|u|^2}
	\qquad\text{for a.e. }u\in\C.
	\]
	The right-hand side belongs to \(L^1(\C,dA)\cap L^2(\C,dA)\).  Hence Plancherel gives
	\[
	H_1^A=H_1^{T_f}
	\qquad\text{a.e. on }\C.
	\]
	Both \(H_1^A\) and \(H_1^{T_f}\) are continuous functions on \(\C\), so the
	equality holds everywhere:
	\[
	H_1^A(w)=H_1^{T_f}(w),
	\qquad w\in\C.
	\]
	
	Apply Lemma~\ref{lem:anti-diagonal-uniqueness} to
	\[
	F(w,z)
	=
	\partial_z
	\bigl(
	\Kern_A(w,z)-\Kern_{T_f}(w,z)
	\bigr).
	\]
	Since
	\[
	F(w,-w)=H_1^A(w)-H_1^{T_f}(w)=0,
	\qquad w\in\C,
	\]
	we obtain
	\[
	F\equiv0
	\qquad\text{on }\C^2.
	\]
	Therefore
	\[
	\partial_z
	\bigl(
	\Kern_A(w,z)-\Kern_{T_f}(w,z)
	\bigr)
	=
	0,
	\qquad w,z\in\C.
	\]
	Thus
	\[
	\Kern_A(w,z)-\Kern_{T_f}(w,z)
	\]
	is independent of \(z\).  It follows from the canonical integral representation
	that \((A-T_f)h\) is constant for every \(h\in\Focka\).  Hence \(A=T_f\) if and
	only if
	\[
	(A-T_f)h(0)=0,
	\qquad h\in\Focka.
	\]
	But by the assumed   the origin functional identity and \(T_f=T_{f_A^{(1)}}\),
	\[
	(A-T_f)h(0)
	=
	Ah(0)-T_{f_A^{(1)}}h(0)
	=
	0,
	\qquad h\in\Focka.
	\]
	Consequently,
	\[
	A=T_f=T_{f_A^{(1)}}.
	\]
	
	It remains to prove uniqueness.  Suppose that \(T_f=T_g\) for two bounded
	symbols \(f,g\in L^\infty(\C,dA)\).  Then \(T_{f-g}=0\).  Applying
	\eqref{eq:H0Toeplitz} to \(f-g\), we obtain
	\[
	\Fsa\!\left(
	\frac\alpha{\pi}(f-g)e^{-\alpha|u|^2}
	\right)=0.
	\]
	Plancherel inversion gives
	\[
	\frac\alpha{\pi}(f-g)e^{-\alpha|u|^2}=0
	\quad\text{a.e.}
	\]
	Therefore \(f=g\) almost everywhere.  The bounded Toeplitz symbol is unique.
\end{proof}

For $m\ge0$, let $\Pol_m$ denote the space of   polynomials of degree at
most $m$.  We set $\Pol_{-1}:=\{0\}$.
\begin{remark}
	The same argument gives the following higher-order version.  For \(n\ge0\), set
	\[
	H_n^A(w)
	=
	\left.\partial_z^n\Kern_A(w,z)\right|_{z=-w}.
	\]
	If \(A=T_f\) with \(f\in L^\infty(\C,dA)\), then
	\[
	H_n^A
	=
	\Fsa\!\left(
	\frac{\alpha^{n+1}}{\pi}
	\bar u^{\,n}f(u)e^{-\alpha|u|^2}
	\right).
	\]
	
	Conversely, let \(n\ge1\).  Assume that \(H_n^A\in L^2(\C,dA)\), and define the
	a.e. defined function on \(\C\) by
	\[
	f_A^{(n)}(u)
	:=
	\frac{1}{\bar u^{\,n}}
	\frac{\pi}{\alpha^{n+1}}
	e^{\alpha|u|^2}
	\bigl(\Fsa^{-1}H_n^A\bigr)(u),
	\qquad u\ne0.
	\]
	If \(f_A^{(n)}\in L^\infty(\C,dA)\), and if \(f\) is any \(L^\infty\)-representative
	of this a.e. equivalence class, then
	\[
	\operatorname{Ran}(A-T_f)\subseteq\Pol_{n-1}.
	\]
	Indeed, the definition of \(f_A^{(n)}\) gives
	\[
	\Fsa^{-1}H_n^A(u)
	=
	\frac{\alpha^{n+1}}{\pi}
	\bar u^{\,n}f(u)e^{-\alpha|u|^2}
	\qquad\text{a.e.}
	\]
	Thus, by Plancherel and the continuity of the anti-diagonal restriction,
	\[
	H_n^A(w)=H_n^{T_f}(w),
	\qquad w\in\C.
	\]
	Applying Lemma~\ref{lem:anti-diagonal-uniqueness} to
	\[
	\partial_z^n
	\bigl(\Kern_A(w,z)-\Kern_{T_f}(w,z)\bigr)
	\]
	yields
	\[
	\partial_z^n
	\bigl(\Kern_A(w,z)-\Kern_{T_f}(w,z)\bigr)
	\equiv0.
	\]
	Hence, for every \(h\in\Focka\), the function \((A-T_f)h\) is a polynomial of
	degree at most \(n-1\).
	
	Consequently, \(A=T_f\) if and only if, for every \(h\in\Focka\),
	\[
	\left.\partial_z^j (A-T_f)h(z)\right|_{z=0}=0,
	\qquad 0\le j\le n-1.
	\]
	Equivalently,
	\[
	\left.
	\partial_z^j
	\bigl(\Kern_A(w,z)-\Kern_{T_f}(w,z)\bigr)
	\right|_{z=0}
	=0,
	\qquad
	w\in\C,\quad 0\le j\le n-1.
	\] 
\end{remark}

We shall use the following complex-parameter version of the Fock reproducing
formula.  When \(c>0\) is real, it is the usual reproducing formula for the Fock
space with Gaussian parameter \(c\).
\begin{lemma} \label{lem:complex-reproducing}
	Let $c,b,z\in\C$ with $\Re c>0$, and let $f\in\Hol(\C)$.  Assume
	\begin{equation}\label{eq:complex-reproducing-integrability}
		\int_\C |f(v)|e^{-(\Re c)|v|^2+\Re(bv+cz\bar v)}\,dA(v)<\infty.
	\end{equation}
	Then
	\begin{equation}\label{eq:complex-reproducing}
		\frac c\pi\int_\C f(v)e^{-c|v|^2+bv+cz\bar v}\,dA(v)
		=
		e^{bz}f(z).
	\end{equation}
\end{lemma}

\begin{proof}
	Set $p(v)=e^{bv}f(v)=\sum_{n\ge0}d_nv^n$.  By
	\eqref{eq:complex-reproducing-integrability}, polar coordinates and Fubini's
	theorem are applicable.  For fixed $r>0$,
	\[
	\int_0^{2\pi}p(re^{i\theta})e^{czre^{-i\theta}}\,d\theta
	=
	2\pi\sum_{n=0}^\infty d_n\frac{(cz)^n}{n!}r^{2n}.
	\]
	The resulting series can be integrated termwise, and
	\[
	\int_0^\infty e^{-cr^2}r^{2n+1}\,dr=\frac{n!}{2c^{n+1}}.
	\]
	Consequently,
	\[
	\int_\C p(v)e^{-c|v|^2+cz\bar v}\,dA(v)
	=
	\frac\pi c\sum_{n=0}^\infty d_nz^n
	=
	\frac\pi c p(z),
	\]
	which is \eqref{eq:complex-reproducing}.
\end{proof}

\begin{lemma} \label{lem:weighted-inversion}
	Let $a,\lambda\in\C$ with $\Re\lambda>0$, and let $\Psi\in\Hol(\C)$.  Assume
	\[
	H(w)=\Psi(-w)e^{\alpha(a\bar w-\lambda|w|^2)}
	\in L^2(\C,dA).
	\]
	Then, in $L^2(\C,dA)$ and hence almost everywhere,
	\begin{equation}\label{eq:weighted-inversion-alpha}
		(\Fsa^{-1}H)(u)
		=\frac\alpha{\pi\lambda}
		\exp\!\left[
		\alpha\frac{a\bar u-|u|^2}{\lambda}
		\right]
		\Psi\!\left(\frac{u-a}{\lambda}\right).
	\end{equation}
\end{lemma}

\begin{proof}
	Put $
	Q(w)=\Psi(-w)$. 
Then \(Q\) is entire, and the hypothesis is equivalent to
	\begin{equation}\label{eq:weighted-fock-norm-proof}
		\int_\C |Q(w)|^2
		e^{-2\alpha \Re\lambda  |w|^2+2\alpha\Re(a\bar w)}\,dA(w)<\infty.
	\end{equation}
	After completing the square, the Hilbert space of entire functions satisfying
	\eqref{eq:weighted-fock-norm-proof} is obtained from a standard Fock space by a
	translation and a dilation.  Hence holomorphic polynomials are dense in this norm,
	and norm convergence implies locally uniform convergence; see
	\cite[Chapter~2]{MR2934601}.  Choose polynomials $Q_j$ such that  $
	Q_j\to Q$ 
	in the norm defined by \eqref{eq:weighted-fock-norm-proof}, and set
	\[
	H_j(w)=Q_j(w)e^{\alpha(a\bar w-\lambda|w|^2)}.
	\]
	Then $H_j\to H$ in $L^2(\C,dA)$.  By Plancherel for $\Fsa$,
	\[
	\Fsa^{-1}H_j\to \Fsa^{-1}H
	\qquad\text{in }L^2(\C,dA).
	\]
	
	For fixed $u\in\C$, the inverse formula gives
	\[
	(\Fsa^{-1}H_j)(u)
	=\frac{\alpha^2}{\pi^2}
	\int_\C Q_j(w)
	e^{-\alpha\lambda|w|^2
		+\alpha\bar u w
		+\alpha(a-u)\bar w}\,dA(w).
	\]
	
	Applying Lemma~\ref{lem:complex-reproducing} with
	\[
	c=\alpha\lambda,\qquad
	b=\alpha\bar u,\qquad
	z=\frac{a-u}{\lambda},
	\]
	we obtain
	\[
	(\Fsa^{-1}H_j)(u)
	=\frac\alpha{\pi\lambda}
	e^{\alpha(a\bar u-|u|^2)/\lambda}
	Q_j\!\left(\frac{a-u}{\lambda}\right).
	\]
	Since $Q_j\to Q$ locally uniformly, the right-hand side converges locally uniformly
	in $u$ to
	\[
	\frac\alpha{\pi\lambda}
	e^{\alpha(a\bar u-|u|^2)/\lambda}
	Q\!\left(\frac{a-u}{\lambda}\right)
	=
	\frac\alpha{\pi\lambda}
	e^{\alpha(a\bar u-|u|^2)/\lambda}
	\Psi\!\left(\frac{u-a}{\lambda}\right).
	\]
	On the other hand, $\Fsa^{-1}H_j\to\Fsa^{-1}H$ in $L^2(\C,dA)$.  Passing to an
	almost-everywhere convergent subsequence identifies the $L^2$-limit with the above
	pointwise limit.  This proves \eqref{eq:weighted-inversion-alpha}.
\end{proof}

\section{Applications}

\subsection{Weighted composition operators}

For $\psi,\varphi\in\Hol(\C)$, define
\[
W_{\psi,\varphi}h=\psi\cdot(h\circ\varphi).
\]
The following standard boundedness criterion for  
\(W_{\psi,\varphi}\), in the present normalization, is known; see
\cite{Ueki} and also \cite[Theorem~3.7]{BFRR}:
\begin{equation}
W_{\psi,\varphi}\in\mathcal L(\Focka)
	\quad\Longleftrightarrow\quad
	\sup_{z\in\C}
 |u(z)|
	e^{\frac\alpha2(|\varphi(z)|^2-|z|^2)}<\infty.
	\nonumber
\end{equation}
If $\psi\not\equiv0$ and $W_{\psi,\varphi}$ is bounded, then
\begin{equation}\label{eq:weighted-affine}
	\varphi(z)=a+\lambda z,
	\qquad |\lambda|\le1.
\end{equation} 

For $\lambda\ne0$, define the candidate symbol
\begin{equation}\label{eq:weighted-candidate-alpha}
	F_{\psi,a,\lambda}(w)
	=
	\frac1\lambda
	\exp\!\left[
	\alpha\left(
	\frac{\lambda-1}{\lambda}|w|^2+\frac a\lambda\bar w
	\right)\right]
	\psi\!\left(\frac{w-a}{\lambda}\right).
\end{equation}

\begin{proof}[Proof of Theorem~\ref{thm:weighted}]
	The assertion is immediate when \(\psi\equiv0\).  Assume henceforth that
	\(\psi\not\equiv0\).
	
	Suppose first that \(W_{\psi,\varphi}=T_f\) with \(f\in L^\infty(\C,dA)\).
	Then \(W_{\psi,\varphi}\) is bounded.  By the standard boundedness theorem for
	weighted composition operators on Fock spaces, there exist \(a,\lambda\in\C\)
	such that
	\[
	\varphi(z)=a+\lambda z,\qquad |\lambda|\le1.
	\]
	If \(\lambda=0\), then
	\[
	H_0^{W_{\psi,\varphi}}(w)
	=
	\psi(-w)e^{\alpha a\bar w}.
	\]
	By Theorem~\ref{thm:higher-recovery}\textup{(a)},
	\(H_0^{W_{\psi,\varphi}}\in L^2(\C,dA)\).  But
	\[
	Q(w):=\psi(-w)e^{\alpha\bar a w}
	\]
	is entire and satisfies
	\[
	|Q(w)|=|H_0^{W_{\psi,\varphi}}(w)|.
	\]
	Thus \(Q\in L^2(\C,dA)\).  The submean inequality forces \(Q\equiv0\), which
	contradicts \(\psi\not\equiv0\).  Hence \(\lambda\ne0\).
	
	For affine \(\varphi(z)=a+\lambda z\), the canonical kernel is
	\[
	\Kern_{W_{\psi,\varphi}}(w,z)
	=
	\psi(z)e^{\alpha(a+\lambda z)\bar w}.
	\]
	Consequently,
	\begin{equation}\label{eq:weighted-fourier-alpha}
		H_0^{W_{\psi,\varphi}}(w)
		=
		\psi(-w)e^{\alpha(a\bar w-\lambda|w|^2)}
		=
		:H(w).
	\end{equation}
	Again by Theorem~\ref{thm:higher-recovery}\textup{(a)}, \(H\in L^2(\C,dA)\).
	Define
	\[
	Q(w)=\psi(-w)e^{\alpha\bar a w}.
	\]
	Then \(Q\) is entire and
	\[
	|H(w)|=|Q(w)|e^{-\alpha(\Re\lambda)|w|^2}.
	\]
	If \(\Re\lambda\le0\), then \(|Q|\le |H|\), so \(Q\in L^2(\C,dA)\).  The
	submean inequality forces \(Q\equiv0\), a contradiction.  Hence $
	\Re\lambda>0$. 
	
	Lemma~\ref{lem:weighted-inversion}, applied to
	\eqref{eq:weighted-fourier-alpha}, gives
	\[
	(\Fsa^{-1}H)(u)
	=
	\frac\alpha{\pi\lambda}
	\exp\!\left[
	\alpha\frac{a\bar u-|u|^2}{\lambda}
	\right]
	\psi\!\left(\frac{u-a}{\lambda}\right).
	\]
	Formula \eqref{eq:recovered-symbol-zero} yields
	\[
	f(u)=F_{\psi,a,\lambda}(u)
	\quad\text{a.e.}
	\]
	In particular, \(F_{\psi,a,\lambda}\in L^\infty(\C,dA)\).
	
	Substituting \(w=a+\lambda z\) in \eqref{eq:weighted-candidate-alpha}, we get
	\begin{equation}\label{eq:weighted-modulus-alpha}
		|F_{\psi,a,\lambda}(a+\lambda z)|
		=
		\frac{e^{\alpha|a|^2}}{|\lambda|}|\psi(z)|
		\exp\!\left[
		\alpha\left(
		(|\lambda|^2-\Re\lambda)|z|^2
		+\Re((2\lambda-1)\bar a z)
		\right)\right].
	\end{equation}
	If \(\Re\lambda<|\lambda|^2\), then the entire function
	\[
	z\longmapsto
	\psi(z)e^{\alpha(2\lambda-1)\bar a z}
	\]
	is bounded by \(Ce^{-\delta|z|^2}\) for some \(\delta>0\), and hence vanishes
	identically.  This contradicts \(\psi\not\equiv0\).  Therefore $
	\Re\lambda\ge|\lambda|^2$.  
	Now \eqref{eq:weighted-modulus-alpha} gives precisely
	\eqref{eq:weighted-multiplier-condition-alpha}.  This proves necessity.
	
	Conversely, assume that
	\[
	\varphi(z)=a+\lambda z,\qquad \lambda\ne0,
	\]
	and that \eqref{eq:weighted-disk-alpha} and
	\eqref{eq:weighted-multiplier-condition-alpha} hold.  Put $
	F=F_{\psi,a,\lambda}$. 
	By \eqref{eq:weighted-modulus-alpha}, the condition
	\eqref{eq:weighted-multiplier-condition-alpha} is exactly the boundedness of
	\(F\). Hence \(F\in L^\infty(\C,dA)\).
	
	We first check that \(W_{\psi,\varphi}\) is bounded.  Let \(M\) denote the
	supremum in \eqref{eq:weighted-multiplier-condition-alpha}.  Then
	\[
	|\psi(z)|
	\le
	M
	\exp\!\left[
	\alpha\left(
	(\Re\lambda-|\lambda|^2)|z|^2
	-\Re((2\lambda-1)\bar a z)
	\right)\right].
	\]
	Therefore
	\begin{align*}
		|\psi(z)|
		e^{\frac{\alpha}{2}(|a+\lambda z|^2-|z|^2)}
		&\le
		M e^{\frac{\alpha}{2}|a|^2}
		\exp\!\left[
		-\frac{\alpha}{2}|1-\lambda|^2|z|^2
		+\alpha\Re((1-\lambda)\bar a z)
		\right].
	\end{align*}
	The right-hand side is bounded on \(\C\). Hence
	\(W_{\psi,\varphi}\in\mathcal L(\Focka)\).
	
	Its anti-diagonal function is
	\[
	H(w)=H_0^{W_{\psi,\varphi}}(w)
	=
	\psi(-w)e^{\alpha(a\bar w-\lambda|w|^2)}.
	\]
	Using \eqref{eq:weighted-multiplier-condition-alpha} with \(z=-w\), we obtain
	\[
	|H(w)|
	\le
	M
	\exp\!\left[
	-\alpha|\lambda|^2|w|^2
	+\alpha\Re(2\lambda\bar a w)
	\right].
	\]
	Since \(\lambda\ne0\), this shows that \(H\in L^2(\C,dA)\).  Moreover,
	\eqref{eq:weighted-disk-alpha} implies \(\Re\lambda>0\).  Lemma
	\ref{lem:weighted-inversion} gives
	\[
	(\Fsa^{-1}H)(u)
	=
	\frac\alpha{\pi\lambda}
	\exp\!\left[
	\alpha\frac{a\bar u-|u|^2}{\lambda}
	\right]
	\psi\!\left(\frac{u-a}{\lambda}\right).
	\]
	Consequently,
	\[
	\frac{\pi}{\alpha}e^{\alpha|u|^2}
	(\Fsa^{-1}H)(u)
	=
	F_{\psi,a,\lambda}(u).
	\]
	Theorem~\ref{thm:higher-recovery}\textup{(a)} now gives
	\[
	W_{\psi,\varphi}=T_{F_{\psi,a,\lambda}}.
	\]
	
	The uniqueness of the bounded symbol follows from
	Theorem~\ref{thm:higher-recovery}.
\end{proof}

\begin{corollary} \label{cor:composition}
	Let $\varphi\in\Hol(\C)$.  Then $C_\varphi=T_f$ for some
	$f\in L^\infty(\C,dA)$ if and only if
	\[
	\varphi(z)=a+\lambda z,
	\qquad \lambda\ne0,
	\]
	and either
	\[
	\Re\lambda>|\lambda|^2,
	\]
	or
	\[
	\Re\lambda=|\lambda|^2
	\quad\text{and}\quad a=0.
	\]
	The unique bounded symbol is
	\begin{equation}\label{eq:composition-symbol-alpha}
		f_{a,\lambda}(w)
		=
		\frac1\lambda
		\exp\!\left[
		\alpha\left(
		\frac{\lambda-1}{\lambda}|w|^2+\frac a\lambda\bar w
		\right)\right]
		\quad\text{a.e.}
	\end{equation}
\end{corollary}

\begin{proof}
	Apply Theorem~\ref{thm:weighted} with $\psi\equiv1$.  
\end{proof}

\subsection{Generalized Volterra--type operators}

For $g,\varphi\in\Hol(\C)$, define
\begin{equation}\label{eq:volterra-definition}
	V_{(g,\varphi)}h(z)
	=
	\int_0^z h(\varphi(\zeta))g'(\zeta)\,d\zeta.
\end{equation}
Equivalently,
\[
V_{(g,\varphi)}=V_g C_\varphi.
\]
The following standard boundedness criterion for the product \(V_gC_\varphi\),
in the present normalization, is known; see \cite{MengestieWorku,MR3192295} and also
\cite[Theorem~3.21]{BFRR}:
\begin{equation}\label{eq:volterra-boundedness-criterion-alpha}
	V_{(g,\varphi)}\in\mathcal L(\Focka)
	\quad\Longleftrightarrow\quad
	\sup_{z\in\C}
	\frac{|g'(z)|}{1+|z|}
	e^{\frac\alpha2(|\varphi(z)|^2-|z|^2)}<\infty.
\end{equation}
If $g'\not\equiv0$ and $V_{(g,\varphi)}$ is bounded, then
\begin{equation}\label{eq:volterra-affine}
	\varphi(z)=a+\lambda z,
	\qquad |\lambda|\le1.
\end{equation} 

For \(a,\lambda\in\C\), \(\lambda\ne0\), define the following a.e. symbol on
\(\C\):
\begin{equation}\label{eq:affine-volterra-candidate-alpha}
	\Sigma_{a,\lambda,g}(w)
	=
	\frac{1}{\alpha\lambda\bar w}
	\exp\!\left[
	\alpha\left(
	\frac{\lambda-1}{\lambda}|w|^2+\frac a\lambda\bar w
	\right)\right]
	g'\!\left(\frac{w-a}{\lambda}\right).
\end{equation}
For $a=0$, write
\begin{equation}\label{eq:centered-volterra-candidate-alpha}
	\sigma_{\lambda,g}(w)
	=
	\frac{1}{\alpha\lambda\bar w}
	\exp\!\left[
	\alpha\frac{\lambda-1}{\lambda}|w|^2
	\right]
	g'\!\left(\frac w\lambda\right).
\end{equation} 

\begin{lemma} \label{lem:affine-volterra-recovery}
	Let $g'\not\equiv0$, $a\in\C$, and $\lambda\ne0$.  If $
	V_{(g,a+\lambda z)}=T_f$ 
	for some $f\in L^\infty(\C,dA)$, then $\Re\lambda>0$ and
	\[
	f=\Sigma_{a,\lambda,g}
	\quad\text{a.e.}
	\]
\end{lemma}

\begin{proof}
	For $w,z\in\C$,
	\[
	\Kern_{V_{(g,a+\lambda z)}}(w,z)
	=
	\int_0^z e^{\alpha(a+\lambda\zeta)\bar w}g'(\zeta)\,d\zeta,
	\]
	and hence
	\begin{equation}\label{eq:affine-volterra-kernel-derivative-alpha}
		\partial_z\Kern_{V_{(g,a+\lambda z)}}(w,z)
		=
		g'(z)e^{\alpha(a+\lambda z)\bar w}.
	\end{equation}
	Therefore
	\begin{equation}\label{eq:affine-volterra-fourier-alpha}
		H_1^{V_{(g,a+\lambda z)}}(w)
		=
		g'(-w)e^{\alpha(a\bar w-\lambda|w|^2)}
		=
		: H_1(w).
	\end{equation}
	By Theorem~\ref{thm:higher-recovery}\textup{(b)}, applied to $T_f$, one has
	$H_1\in L^2(\C,dA)$.  If $\Re\lambda\le0$, then
	\[
	g'(-w)e^{\alpha\bar a w}\in L^2(\C,dA),
	\]
	and the submean inequality forces this entire function to vanish identically, a
	contradiction.  Thus $\Re\lambda>0$.
	
	Lemma~\ref{lem:weighted-inversion}, with $\Psi=g'$, gives
	\[
	(\Fsa^{-1}H_1)(u)
	=
	\frac\alpha{\pi\lambda}
	e^{\alpha(a\bar u-|u|^2)/\lambda}
	g'\!\left(\frac{u-a}{\lambda}\right).
	\]
	By \eqref{eq:first-order-recovered-symbol}, the first-order recovered symbol is
	\[
	f_{V_{(g,a+\lambda z)}}^{(1)}(u)
	=
	\frac{1}{\alpha\lambda\bar u}
	\exp\!\left[
	\alpha\left(
	\frac{\lambda-1}{\lambda}|u|^2+\frac a\lambda\bar u
	\right)\right]
	g'\!\left(\frac{u-a}{\lambda}\right)
	\]
	for a.e. \(u\in\C\setminus\{0\}\). 
	Since \(V_{(g,a+\lambda z)}=T_f\), Theorem~\ref{thm:higher-recovery}
	\textup{(b)} and \eqref{eq:affine-volterra-candidate-alpha} give  
	\[
	f=\Sigma_{a,\lambda,g}
	\qquad\text{a.e.}
	\]
\end{proof}

\begin{lemma} \label{lem:volterra-origin}
	Let $g'\not\equiv0$, $a\in\C$, and $\lambda\ne0$.  Assume $
	\Sigma_{a,\lambda,g}\in L^\infty(\C,dA)$. 
	Then
	\begin{equation}\label{eq:affine-zero-condition-alpha}
		g'\!\left(-\frac a\lambda\right)=0,
		\qquad
		\Re\lambda\ge|\lambda|^2>0.
	\end{equation}
	Moreover, for every \(h\in\Focka\),
	\begin{equation}\label{eq:origin-evaluation-formula-alpha}
		T_{\Sigma_{a,\lambda,g}}h(0)
		=
		\frac1\lambda\int_0^a
		g'\!\left(\frac{t-a}{\lambda}\right)h(t)\,dt.
	\end{equation}
	Here the integral is taken over the straight line segment from \(0\) to \(a\).
	Since the integrand is entire in \(t\), the value is path independent.
\end{lemma}

\begin{proof}
	The function in \eqref{eq:affine-volterra-candidate-alpha} is continuous on
	$\C\setminus\{0\}$.  If $g'(-a/\lambda)\ne0$, then its modulus is bounded below
	by a positive multiple of $|w|^{-1}$ on a punctured neighborhood of the origin,
	contradicting essential boundedness.  Hence $g'(-a/\lambda)=0$.
	
	Substituting $w=a+\lambda v$ in
	\eqref{eq:affine-volterra-candidate-alpha} gives
	\begin{align}
		|\Sigma_{a,\lambda,g}(a+\lambda v)|
		&=
		\frac{e^{\alpha|a|^2}}
		{\alpha|\lambda|\,|a+\lambda v|}|g'(v)|
		\notag 
		\exp\!\left[
		\alpha\left(
		(|\lambda|^2-\Re\lambda)|v|^2
		+\Re((2\lambda-1)\bar a v)
		\right)\right].
		\label{eq:affine-volterra-modulus-alpha}
	\end{align}
	If $\Re\lambda<|\lambda|^2$, then the entire function $
	g'(v)e^{\alpha(2\lambda-1)\bar a v}$ 
	is bounded by $C(1+|v|)e^{-\delta|v|^2}$ for some $\delta>0$.  It is therefore
	bounded and tends to zero at infinity, so Liouville's theorem forces it to vanish
	identically.  This contradicts $g'\not\equiv0$.  Thus
	$\Re\lambda\ge|\lambda|^2>0$.
	
	Fix $h\in\Focka$ and put
	\[
	R(w)=g'\!\left(\frac{w-a}{\lambda}\right)h(w).
	\]
	Then $R$ is entire and $R(0)=0$.  For $\eta\in\C$, define
	\[
	I_h(\eta)
	=
	\frac1{\pi\lambda}\int_\C
	\frac{R(w)}{\bar w}
	e^{-\alpha|w|^2/\lambda+\alpha\eta\bar w/\lambda}\,dA(w).
	\]
	The integral is locally uniformly absolutely convergent in $\eta$.  Indeed, using
	\eqref{eq:affine-volterra-candidate-alpha}, its absolute integrand is bounded by a
	constant times
	\[
	|\Sigma_{a,\lambda,g}(w)|\,|h(w)|
	e^{-\alpha|w|^2+
		\Re(\alpha(\eta-a)\bar w/\lambda)},
	\]
	which is integrable by Cauchy--Schwarz, locally uniformly for $\eta$ in compact
	sets.  Differentiation under the integral sign gives
	\[
	I_h'(\eta)
	=
	\frac\alpha{\pi\lambda^2}\int_\C
	R(w)e^{-\alpha|w|^2/\lambda+\alpha\eta\bar w/\lambda}\,dA(w).
	\]
	Apply Lemma~\ref{lem:complex-reproducing} with
	$c=\alpha/\lambda$, $b=0$, and $f=R$.  Since
	$\Re(1/\lambda)>0$, we obtain
	\begin{equation}\label{eq:I-derivative-alpha}
		I_h'(\eta)=\frac1\lambda R(\eta).
	\end{equation}
	
	To compute $I_h(0)$, write $R(w)=wS(w)$ with $S$ entire.  On each circle, all
	angular Fourier frequencies of
	\[
	\frac w{\bar w}S(w)e^{-\alpha|w|^2/\lambda}
	\]
	are strictly positive.  Absolute convergence permits Fubini's theorem, and hence
	$I_h(0)=0$.  Integrating \eqref{eq:I-derivative-alpha} from $0$ to $a$ gives
	\[
	I_h(a)=\frac1\lambda\int_0^aR(t)\,dt.
	\]
	Finally, the factors $\alpha/\pi$ in the Toeplitz integral and $1/\alpha$ in
	the candidate symbol cancel, so
	\[
	I_h(a)=T_{\Sigma_{a,\lambda,g}}h(0).
	\]
\end{proof}

\begin{proof}[Proof of Theorem~\ref{thm:volterra-main}]
	Assume first that \(V_{(g,\varphi)}=T_f\) with \(f\in L^\infty(\C,dA)\).  If
	\(g\) is constant, then \(V_{(g,\varphi)}=0\).  Assume now that
	\(g'\not\equiv0\).  Since \(V_{(g,\varphi)}\) is bounded, the standard
	boundedness theorem for generalized Volterra--composition operators on Fock
	spaces gives
	\[
	\varphi(z)=a+\lambda z,\qquad |\lambda|\le1.
	\]
	If \(\lambda=0\), then
	\[
	H_1^{V_{(g,\varphi)}}(w)
	=
	g'(-w)e^{\alpha a\bar w}.
	\]
	By Theorem~\ref{thm:higher-recovery}\textup{(b)},
	\(H_1^{V_{(g,\varphi)}}\in L^2(\C,dA)\).  But
	\[
	Q(w):=g'(-w)e^{\alpha\bar a w}
	\]
	is entire and satisfies
	\[
	|Q(w)|=|H_1^{V_{(g,\varphi)}}(w)|.
	\]
	Thus \(Q\in L^2(\C,dA)\).  The submean inequality forces \(Q\equiv0\), which
	contradicts \(g'\not\equiv0\).  Hence \(\lambda\ne0\).
	
	Lemma~\ref{lem:affine-volterra-recovery} gives
	\[
	f=\Sigma_{a,\lambda,g}
	\qquad\text{a.e.}
	\]
	In particular, \(\Sigma_{a,\lambda,g}\in L^\infty(\C,dA)\).  Since
	\(V_{(g,\varphi)}h(0)=0\), Lemma~\ref{lem:volterra-origin} yields
	\begin{equation}\label{eq:all-test-functions-alpha}
		\int_0^a g'\!\left(\frac{t-a}{\lambda}\right)h(t)\,dt=0,
		\qquad h\in\Focka.
	\end{equation}
	We claim that \(a=0\).  Suppose, to the contrary, that \(a\ne0\).  Taking
	\(h(t)=p(t/a)\), where \(p\) is an arbitrary polynomial, and using the
	straight-line parametrization \(t=as\), \(0\le s\le1\), we obtain
	\[
	\int_0^1
	g'\!\left(\frac{a(s-1)}{\lambda}\right)p(s)\,ds=0
	\]
	for every polynomial \(p\).  By the Weierstrass approximation theorem, the same
	identity holds for every continuous function on \([0,1]\).  Applying this to
	continuous functions approximating
	\[
	\overline{g'\!\left(\frac{a(s-1)}{\lambda}\right)}
	\]
	uniformly on \([0,1]\), we get
	\[
	\int_0^1
	\left|
	g'\!\left(\frac{a(s-1)}{\lambda}\right)
	\right|^2\,ds=0.
	\]
	Thus \(g'\) vanishes on the  line segment
	\[
	\left\{
	\frac{a(s-1)}{\lambda}:0\le s\le1
	\right\}.
	\]
	By the identity theorem, \(g'\equiv0\), a contradiction.  Hence \(a=0\).
	Therefore
	\[
	\varphi(z)=\lambda z,
	\qquad
	f=\sigma_{\lambda,g}
	\quad\text{a.e.}
	\]
	The automatic conditions
	\[
	g'(0)=0,\qquad \Re\lambda\ge|\lambda|^2>0
	\]
	follow from Lemma~\ref{lem:volterra-origin} with \(a=0\).
	
	Conversely, suppose that \(g'\not\equiv0\),
	\[
	\varphi(z)=\lambda z,\qquad \lambda\ne0,
	\]
	and $\sigma_{\lambda,g}\in L^\infty(\C,dA)$.
	 
	Lemma~\ref{lem:volterra-origin}, with \(a=0\), gives
	\[
	g'(0)=0,\qquad \Re\lambda\ge|\lambda|^2>0.
	\]
	From \eqref{eq:centered-volterra-candidate-alpha}, we obtain
	\begin{equation}\label{g'}
	|g'(z)|
\le
C|z|e^{\alpha(\Re\lambda-|\lambda|^2)|z|^2}.
	\end{equation}
	Consequently,
	\begin{align*}
		\frac{|g'(z)|}{1+|z|}
		e^{\frac\alpha2(|\lambda z|^2-|z|^2)}
		&\le
		C\exp\!\left[
		\frac\alpha2(2\Re\lambda-|\lambda|^2-1)|z|^2
		\right]   =
		C e^{-\frac\alpha2|1-\lambda|^2|z|^2}.
	\end{align*}
	Thus \(V_{(g,\lambda z)}\) is bounded by
	\eqref{eq:volterra-boundedness-criterion-alpha}.
	
	Its first-order anti-diagonal restriction is
	\[
	H_1(w)=g'(-w)e^{-\alpha\lambda|w|^2}.
	\]
	The  growth estimate \eqref{g'} gives
	\[
	|H_1(w)|
	\le
	C|w|e^{-\alpha|\lambda|^2|w|^2},
	\]
	so \(H_1\in L^2(\C,dA)\).  Lemma~\ref{lem:weighted-inversion}, with \(a=0\) and
	\(\Psi=g'\), gives
	\[
	(\Fsa^{-1}H_1)(u)
	=
	\frac{\alpha}{\pi\lambda}
	e^{-\alpha|u|^2/\lambda}
	g'\!\left(\frac u\lambda\right).
	\]
	Therefore, by \eqref{eq:first-order-recovered-symbol},
	\[
	f_{V_{(g,\lambda z)}}^{(1)}(u)
	=
	\sigma_{\lambda,g}(u)
	\qquad\text{for a.e. }u\in\C\setminus\{0\}.
	\]
	Moreover,
	\[
	V_{(g,\lambda z)}h(0)=0,
	\]
	and Lemma~\ref{lem:volterra-origin}, with \(a=0\), gives
	\[
	T_{\sigma_{\lambda,g}} h(0)=0.
	\]
	Thus  the origin functional identity in
	Theorem~\ref{thm:higher-recovery}\textup{(b)} holds. Hence
	\[
	V_{(g,\lambda z)}=T_{\sigma_{\lambda,g}}.
	\]
	The case \(g\) constant is trivial, since \(V_{(g,\varphi)}=0\).
\end{proof}
 
\begin{corollary}\label{cor:classical-volterra}
	Let
	\[
	V_gh(z)=\int_0^z h(\zeta)g'(\zeta)\,d\zeta.
	\]
	Then $
	V_g=T_f$ for some $f\in L^\infty(\C,dA)$
 	if and only if $g(z)=a+cz^2$.  In this case
	\[
	f(z)=\frac{2c}{\alpha}\frac z{\bar z},
	\qquad z\ne0.
	\]  
\end{corollary}

\begin{proof}
	Apply Theorem~\ref{thm:volterra-main} with $\lambda=1$.  The boundedness
	characterization follows from \cite{Constantin} at $\alpha=1$ and
	Lemma~\ref{lem:dilation}. 
\end{proof}
\begin{remark}
	For \(\alpha=1\), \cite[Proposition~3.23]{BFRR} gives the nontrivial
	example corresponding to \(g(z)=z^2/2\).  Corollary~\ref{cor:classical-volterra}
	shows that, up to multiplication by a constant and addition of a constant to
	\(g\), this example exhausts all classical Volterra operators which are
	Toeplitz operators with bounded measurable symbols.
\end{remark}
	
	\subsection*{Conflict of interest}
	The authors have no conflict of interest to declare that are relevant to the content
	of this article.
	
	\subsection*{Data availability statement}
	No data, models, or code were generated or used for the research described in the
	article.
	
	\bibliographystyle{amsplain}
	\bibliography{references}
	
	\medskip
	
	\noindent
	School of Mathematical Sciences, Dalian University of Technology,
	Dalian, Liaoning 116024, P. R. China
	
	\noindent
	Email address: \texttt{zhaopeng.lin@mail.dlut.edu.cn}
	
	\medskip
	
	\noindent
	School of Mathematical Sciences, Dalian University of Technology,
	Dalian, Liaoning 116024, P. R. China
	
	\noindent
	Email address: \texttt{lyfdlut@dlut.edu.cn}
	
	\medskip
	
	\noindent
	School of Mathematical Sciences, Dalian University of Technology,
	Dalian, Liaoning 116024, P. R. China
	
	\noindent
	Email address: \texttt{zuchao@dlut.edu.cn}
	
\end{document}